\documentclass[11pt]{amsart}
\usepackage{amssymb, latexsym, graphicx, mathrsfs, enumerate, amsmath, amsthm, textcomp, url, tikz-cd,bbm}
\usepackage[T1]{fontenc}
\usepackage{mathdots}
\usepackage{comment}
\usepackage[toc,page]{appendix}
\usepackage{hyperref}
\hypersetup{
    colorlinks=true,
    linkcolor=blue,
    filecolor=magenta,      
    urlcolor=cyan,
}

\usepackage{color}

\input xy
\xyoption{all}

\allowdisplaybreaks
\allowdisplaybreaks

\numberwithin{equation}{section}
\newtheorem{theorem}{Theorem}[section]
\newtheorem{proposition}[theorem]{Proposition}

\newtheorem{lemma}[theorem]{Lemma}

\newtheorem*{remark*}{Remark}

\theoremstyle{definition}
\newtheorem{definition}[theorem]{Definition}

\newtheorem{remark}[theorem]{Remark}
\newcommand{\vpi}{\varphi}
\newcommand{\mfo}{\mathfrak{o}}

\newcommand{\M}{\mathrm{M}}

\newcommand{\End}{\mathrm{End}}
\newcommand{\cL}{\mathcal{L}}

\newcommand{\ord}{\mathrm{ord}}
\newcommand*{\Rom}[1]{\expandafter\@slowromancap\romannumeral #1@}
\begin{document}

\title[Orbital integrals  
for the spherical Hecke algebra of $\mathrm{GL}_3$]
{An explicit formula for the orbital integrals  
on the spherical Hecke algebra of $\mathrm{GL}_3$}



\author[Sungmun Cho]{Sungmun Cho}
\author[Yuchan Lee]{Yuchan Lee}
\thanks{The authors are supported by  Samsung Science and Technology Foundation under Project Number SSTF-BA2001-04.}

\address{Sungmun Cho \\  Department of Mathematics, POSTECH, 77, Cheongam-ro, Nam-gu, Pohang-si, Gyeongsangbuk-do, 37673, KOREA}

\email{sungmuncho12@gmail.com}

\address{Yuchan Lee \\  Department of Mathematics, POSTECH, 77, Cheongam-ro, Nam-gu, Pohang-si, Gyeongsangbuk-do, 37673, KOREA}

\email{yuchanlee329@gmail.com}

\maketitle

\begin{abstract}
We provide the explicit formula for orbital integrals associated with elliptic regular semisimple elements in $\mathrm{GL}_{n}(F)\cap \mathrm{M}_n(\mfo)$  and associated with  arbitrary elements of the spherical Hecke algebra of $\mathrm{GL}_{n}(F)$ when $n=2, 3$, using results of  \cite{CKL}. Here $F$ is a non-Archimedean local field of any characteristic with $\mfo$ its ring of integers.
\end{abstract}

\tableofcontents

\section{Introduction}


Let $F$ be a non-Archimedean local field of any characteristic and let $\mfo$ be its ring of integers with $\kappa$ its residue field and with $\pi$  a uniformizer. The spherical Hecke algebra $\mathcal{H}$ of $\mathrm{GL}_n(F)$ is the space of complex-valued $\mathrm{GL}_n(\mfo)$-biinvariant continuous functions of compact support on $\mathrm{GL}_n(F)$. 
The spherical Hecke algebra $\mathcal{H}$ admits a $\mathbb{C}$-linear basis which consists of the characteristic function on $D_{(k_1,\cdots,k_n)}$, denoted by $\mathbbm{1}_{D_{(k_1,\cdots,k_n)}}$, where $k_i\in \mathbb{Z}$ with $k_1\leq \cdots \leq k_n$ and where 
\[
D_{(k_1,\cdots,k_n)}=\mathrm{GL}_n(\mfo)\cdot 
\begin{pmatrix}
    \pi^{k_1}&&\\
    &\ddots&\\
    &&\pi^{k_n}
\end{pmatrix}
\cdot \mathrm{GL}_n(\mfo).
\]

For a regular and semisimple element $\gamma \in \mathrm{GL}_{n}(F)\cap \mathrm{M}_n(\mfo)$ and for an element $\mathbbm{1}_{D_{(k_1,\cdots,k_n)}}\in \mathcal{H}$, 
 the orbital integral, denoted by $\mathcal{SO}_{\gamma}(\mathbbm{1}_{D_{(k_1,\cdots,k_n)}})$,  is defined to be an integration of $\mathbbm{1}_{D_{(k_1,\cdots,k_n)}}$ on the conjugacy class of $\gamma$ with respect to a certain Haar masure (cf. Definition \ref{def:orbitalint}). 
In this paper, we propose the closed formula for $\mathcal{SO}_{\gamma}(\mathbbm{1}_{D_{(k_1,\cdots,k_n)}})$ when $\gamma$ is elliptic and when $n=2,3$, based on  results of \cite{CKL}.

\subsection{Background}

The explicit formula for  $\mathcal{SO}_{\gamma}(\mathbbm{1}_{D_{(k_1,\cdots,k_n)}})$ is known only for  $\mathrm{GL}_2$ by Kottwitz in \cite{Kot05}.
His method is to  interpret $\mathcal{SO}_{\gamma}(\mathbbm{1}_{D_{(k_1, k_2)}})$ as a counting problem on  the Bruhat-Tits building of $\mathrm{SL}_2$.
This method is hardly applicable for $\mathrm{GL}_n$ with $n>2$ since  the structure of the  Bruhat-Tits building of $\mathrm{SL}_n$ with $n>2$ is complicated. 
On the other hand,  Kottwitz in \cite{Kot80} described $\mathcal{SO}_{\gamma}(\mathbbm{1}_{D_{(k_1,k_2,k_3)}})$ for $\mathrm{GL}_3$  as  linear combinations of more basic linear functionals, $L_1^m,L_2^m$ and $L_3^m$, on the spherical Hecke algebra $\mathcal{H}$, using the Bruhat-Tits building of $\mathrm{SL}_3$.

When  the characteristic of $F$ is greater than $5$, Chen in \cite{CHEN} obtained the formula for  orbital integrals for an elliptic regular semisimple element $\gamma$ in $\mathrm{GL}_3(F)\cap  \mathrm{M}_3(\mfo)$ and for $\mathbbm{1}_{\mathrm{M}_3(\mfo)}$, where $\mathbbm{1}_{\mathrm{M}_3(\mfo)}$ is the characteristic function on $\mathrm{M}_3(\mfo)$.
His method is to count the number of rational points on the truncated affine Springer fibers. 


Recently in \cite{CKL}, the authors and Kang proposed another method to investigate  orbital integrals using smoothening of a certain scheme over $\mfo$, and provided
the closed formula for $\mathcal{SO}_{\gamma}(\mathbbm{1}_{M_3(\mfo)})$
where $\gamma$ is a regular semisimple element in $\mathrm{GL}_3(F)\cap \mathrm{M}_3(\mfo)$  (see Proposition \ref{thm_for_M_3}).
One of key steps in this method is to use a stratification of the conjugacy class of $\gamma$. 
Since this is also used in our argument, 
we  describe it precisely below.

Let us identify $\mathrm{M}_3(\mfo)$ with $\mathrm{End}(L)$ for a free $\mfo$-module $L$ of rank $3$.
For a lattice $M\subset L\otimes_{\mfo}F$ of rank $3$, 
the type of $M$ is defined to 
    be $(k_1,k_2,k_3)$ where $k_i\in \mathbb{Z}$ with $ k_1\leq k_2\leq k_3$, such that $M\subset \pi^{k_1} L$ satisfying
    \[
    (\pi^{k_1} L)/M\cong \mfo/\pi^{k_1-k_1}\mfo\oplus\mfo/\pi^{k_2-k_1}\mfo \oplus \mfo/\pi^{k_3-k_1}\mfo.
    \]

Let  $G_{\gamma}(F):=\{ g^{-1}\gamma g\in \mathrm{M}_3(F) \mid g\in \mathrm{GL}_3(F)\}$ and let $d=\mathrm{ord}(\mathrm{det}(\gamma))$.
Then we have the following stratification for $G_\gamma (F)\cap \mathrm{M}_3(\mfo)$ (cf. \cite[Proposition 3.7]{CKL} or Proposition \ref{prop_stratification}):
\begin{equation}\label{intro_stratification}
G_\gamma(F)\cap \mathrm{M}_3(\mfo)
=\bigsqcup_{\substack{(k_1, k_2, k_3),  \\ \sum k_i=d,\ k_1\geq 0 }}
\Bigg(\bigsqcup_{\substack{M:type(M)= \\ (k_1, k_2, k_3)}} O_{\gamma, \cL(L,M)}\Bigg),
\end{equation}
where
$\left\{
\begin{array}{l}
     \cL(L,M)(\mfo)=\{f\in \mathrm{End}_\mfo(L) \mid f:L \rightarrow M \textit{ is surjective}\}; \\
     \mathcal{O}_{\gamma,\cL(L,M)}=G_\gamma(F)\cap \cL(L,M)(\mfo).
\end{array}
\right.
$

\subsection{Main results}
Our main result is the closed formula for $\mathcal{SO}_{\gamma}(\mathbbm{1}_{D_{(k_1,\cdots,k_n)}})$ when $\gamma$ is elliptic, which is stated in Theorem \ref{mainthm48} when $n=2$ and in Theorem \ref{mainthmgl3} when $n=3$.
Since Theorem \ref{mainthm48} with $n=2$ is rather simple and  is not new (the formula was known by Kottwitz), we will  explain outline of our proof when $n=3$.

The starting point of our proof is the following identity proved in Proposition \ref{prop:hecke_type}:
\[ \mathcal{SO}_\gamma(\mathbbm{1}_{D_{(k_1,k_2,k_3)}})=\textit{the volume of $\bigsqcup\limits_{\substack{M:type(M)= \\ (k_1, k_2, k_3)}} O_{\gamma, \cL(L,M)}$}.\]
By the reduction formula stated in \cite[Proposition 4.11]{CKL} (see Proposition \ref{prop_red}), we can  suppose that $k_1=0$.
If we exclude the case that $3|d$ and $(k_1,k_2,k_3)=(0,\frac{d}{3},\frac{2d}{3})$, then the  explicit formula for the volume of $\bigsqcup\limits_{\substack{M:type(M)= \\ (0, k_2, k_3)}} O_{\gamma, \cL(L,M)}$ is provided in \cite[Theorem 6.1]{CKL} which directly yields  the formula for $\mathcal{SO}_{\gamma}(\mathbbm{1}_{D_{(0,k_2,k_3)}})$ (see Proposition \ref{result_3}).
\cite[Theorem 6.1]{CKL} also gives the formula for   $\mathcal{SO}_{\gamma}(\mathbbm{1}_{D_{(0,\frac{d}{3},\frac{2d}{3})}})$ under the restriction that the reduction of $\chi_\gamma(\pi^{\frac{d}{3}}x)/\pi^d$ modulo $\pi$ is irreducible over $\kappa$, where $\chi_{\gamma}(x)$ is the characteristic polynomial of $\gamma$ and $d=\mathrm{ord}(\mathrm{det}(\gamma))$.


To overcome this restriction when $3| d$, 
we use \cite[Proposition 5.1]{CKL} which shows the existence of a constant $a\in \mfo$ such that if $3|d_a$ with $d_a=\mathrm{ord}(\mathrm{det}(\gamma-a\cdot I_3))$ then the characteristic polynomial of $\gamma -a\cdot I_3$ satisfies the desired condition that $\chi_{\gamma -a\cdot I_3}(\pi^{\frac{d_a}{3}}x)/\pi^{d_a}$ modulo $\pi$ is irreducible over $\kappa$ (see Lemma \ref{prop_specialconst}.(1)). 
Here $I_3$ denotes the identity matrix in $\mathrm{M}_3(\mfo)$. 
Then \cite[Lemma 3.4]{CKL} yields the following   bijection  (see Lemma \ref{constantlemma}):
$$G_\gamma(F)\cap \mathrm{M}_3(\mfo)\longrightarrow  G_{\gamma-a\cdot I_3}(F)\cap \mathrm{M}_3(\mfo), ~~~~  
  \ X\mapsto X-a\cdot I_3.$$ 
We prove that this bijection induces the following identity in the paragraph following Equation (\ref{eq:special_type1}): 
\[
\sum_{j=1}^{\lfloor d/2 \rfloor}\mathcal{SO}_{\gamma}(\mathbbm{1}_{D_{(0,j,d-j)}})=
    \sum_{i=1}^{\lfloor d_a/2 \rfloor}\mathcal{SO}_{\gamma-a\cdot I_3}(\mathbbm{1}_{D_{(0,i,d_a -i)}}).
\]

On the left hand side of the above identity, we obtain the explicit formula for each summand except for $\mathcal{SO}_{\gamma}(\mathbbm{1}_{D_{(0,\frac{d}{3},\frac{2d}{3})}})$ by \cite[Theorem 6.1]{CKL}. 
On the other hand, each summand of the right hand side, especially $\mathcal{SO}_{\gamma}(\mathbbm{1}_{D_{(0,\frac{d_a}{3},\frac{2d_a}{3})}})$ if $3|d_a$ by the characterization of $a\in \mfo$,  is explicitly obtained by \cite[Theorem 6.1]{CKL} as well.  
Comparing these two sides, we finally 
 obtain the formula for $\mathcal{SO}_{\gamma}(\mathbbm{1}_{D_{(0,\frac{d}{3},\frac{2d}{3})}})$.

\begin{remark}
There are  two Haar measures with respect to which an orbital integral is defined; the geometric measure and the quotient measure.  These two measures are described in Section \ref{subsection:geometric} and Section \ref{classical_measure} respectively. The  difference between them  is described in \cite[Proposition 3.29]{FLN} (see Proposition \ref{proptrans}). 

In this manuscript, we work with the geometric measure.
However, 
we translate   Theorem \ref{mainthm48} and Theorem \ref{mainthmgl3}   in terms of the quotient measure using Proposition \ref{proptrans} in  Theorem \ref{result_2_dmu} and Theorem \ref{mainthmgl3q} respectively. 
\end{remark}



\textbf{Organizations.}
The structure of the paper is as follows. After fixing notations in Section \ref{sectionnss}, we 
explain our measures and orbital integrals explicitly  in Section \ref{sec3}. 
In Section \ref{sec4}, after collecting  a few results from \cite{CKL}, we provide the formula for $\mathcal{SO}_{\gamma}(\mathbbm{1}_{D_{(k_1,k_2)}})$ with $n=2$ in Section \ref{subsec4.1} and the formula for $\mathcal{SO}_{\gamma}(\mathbbm{1}_{D_{(k_1,k_2,k_3)}})$ with $n=3$ in Section \ref{subsec4.2}. 
\\

\textbf{Acknowledgments.}
We would like to thank Masao Tsuzuki to suggest this problem and to explain applications. 
We appreciate   Taeyeoup Kang and Sandeep Varma for  helpful discussions.

\section{Notations}\label{sectionnss}
This section is taken from \cite[Notations]{CKL}.
\begin{itemize}
\item Let $F$ be a  non-Archimedean local field  of any characteristic with $\mathfrak{o}_F$  its ring of integers and $\kappa$  its residue field.
Let $\pi$ be a uniformizer in $\mathfrak{o}_F$.
Let $q$ be the cardinality of the finite field $\kappa$.
If there is no confusion, we sometimes use $\mathfrak{o}$ to stand for $\mathfrak{o}_F$. \\
More generally, for a finite field extension $F'$ of $F$, the ring of integers in $F'$ is denoted by $\mfo_{F'}$ and the residue field of $\mfo_{F'}$ is denoted by $\kappa_{F'}$.

\item For an element $x\in F$, the exponential order of $x$ with respect to the maximal ideal in $\mathfrak{o}$ is written by $\mathrm{ord}(x)$.

\item For an element $x\in F$, the value of $x$ is $|x|_F:=q^{-\mathrm{ord}(x)}$.
If there is no confusion, then we sometimes omit $F$ so that the value of $x$ is written as $|x|$.

\item Let $\mathrm{GL}_{n, A}$ be  the general linear group scheme defined over $A$ 
  and let $\mathbb{A}^n_A$ be the affine space of dimension $n$ defined over $A$,
   where $A$ is a commutative $\mfo$-algebra. 
If there is no confusion then we sometimes omit $A$ in the subscript to express schemes over $\mfo$. 
Thus $\mathrm{GL}_{n}$  stands for $\mathrm{GL}_{n, \mathfrak{o}}$. 
Similarly $\mathrm{M}_n$ is the scheme over $\mfo$ representing the set of $n \times n$ matrices.



\item For $a\in R$ or $f(x)\in R[x]$ with a flat $\mfo$-algebra $R$, $\bar{a}\in R\otimes \kappa$ or $\bar{f}(x)\in R\otimes \kappa[x]$ is the reduction of $a$ or $f(x)$ modulo $\pi$, respectively.


\item
For $\gamma \in \mathrm{GL}_{n}(F)$,
 let  $\chi_{\gamma}(x)\in F[x]$  be its characteristic polynomial.
We always write $$\chi_{\gamma}(x)=x^n+c_1x^{n-1}+\cdots + c_{n-1}x+c_n ~~~~ \textit{with $c_i\in F$}. $$

\item Let $\Delta_{\gamma}$ be the discriminant of $\chi_{\gamma}(x)$.
 
\item By saying $\gamma\in \mathrm{GL}_{n}(F)$ regular, we mean that the identity component of the centralizer of $\gamma$ in $\mathrm{GL}_{n, F}$ is a maximal torus.
 In particular,  $\gamma$ of being  regular and semisimple  is equivalent that $\chi_{\gamma}(x)$ has distinct roots in the algebraic closure of $F$,
equivalently  $\Delta_{\gamma}\neq 0$ (cf. \cite[Section 3.1]{Gor22} or \cite[Sections 6-7]{Gro05}).


 \item Let $\mathcal{H}$ be the spherical Hecke algebra of $\mathrm{GL}_n(F)$, which is defined to be the space of complex-valued $\mathrm{GL}_n(\mfo)$-biinvariant  continuous functions of compact support on $\mathrm{GL}_n(F)$. 

\item Let $D_{(k_1,\cdots,k_n)}:=\mathrm{GL}_n(\mfo)\cdot\mathrm{diag}(\pi^{k_1}, \cdots, \pi^{k_n})\cdot \mathrm{GL}_n(\mfo)$, where $k_i\in \mathbb{Z}$ with $k_1\leq \cdots \leq k_n$.
Here $\mathrm{diag}(\pi^{k_1}, \cdots, \pi^{k_n})$ is the diagonal matrix of size $n$ with $\pi^{k_i}$ as the $i$-th diagonal entry. 
Then Cartan decomposition yields that $\mathbbm{1}_{D_{(k_1,\cdots,k_n)}}$, which is the characteristic function on $D_{(k_1,\cdots,k_n)}$, forms a basis for $\mathcal{H}$ as a $\mathbb{C}$-vector space.

\item For a rational number $a\in \mathbb{Q}$, 
$\lfloor a\rfloor$ is the largest integer which is less than or equal to $a$.
and  $\lceil a \rceil$ is the smallest integer which is greater than or equal to $a$.

\end{itemize}

From now on until the end of this paper, we suppose that $\gamma $ is a regular and semisimple element in $\mathrm{GL}_{n}(F)\cap \mathrm{M}_n(\mfo)$. 
Here  $\mathrm{GL}_{n}$ is defined over $\mfo$ and thus  is an open subscheme of $\mathrm{M}_{n}$ over $\mfo$. The intersection  $\mathrm{GL}_{n}(F)\cap \mathrm{M}_n(\mfo)$ is taken inside $\mathrm{M}_n(F)$.
Note that $\chi_{\gamma}[x]\in \mfo[x]$.

\section{Definition of orbital integrals}\label{sec3}
In this section, we define the orbital integral for a regular and semisimple element $\gamma \in \mathrm{GL}_{n}(F)\cap \mathrm{M}_n(\mfo)$ and for an element $\mathbbm{1}_{D_{(k_1,\cdots,k_n)}}$ in the  spherical Hecke algebra.

\subsection{Measures}
We will first introduce the geometric measure following \cite{FLN} and then the quotient measure following \cite{Yun13} which is often used in the literature. 
The comparison between them will also be explained.
Our arguments follow \cite[Section 2.1]{CKL}.

\subsubsection{Geometric measure}\label{subsection:geometric}
Let $\omega_{\mathrm{M}_n, \mfo}$ and $\omega_{\mathbb{A}_\mfo^n}$ be nonzero, translation-invariant forms on $\mathrm{M}_{n,F}$ and $\mathbb{A}_F^n$, respectively, with normalizations
\[
\int_{\mathrm{M}_{n, \mathfrak{o}}(\mathfrak{o})}|\omega_{\mathrm{M}_{n, \mathfrak{o}}}|=1 \mathrm{~and~}  \int_{\mathbb{A}^n_{\mathfrak{o}}(\mathfrak{o})}|\omega_{\mathbb{A}^n_{\mathfrak{o}}}|=1.
\]

Define a map 
\[
\rho_n : \mathrm{M}_{n,F} \longrightarrow \mathbb{A}_F^n, ~~~ \gamma\mapsto 
\textit{coefficients of $\chi_{\gamma}(x)$}.
\]
That is,  $\rho_n(\gamma)=(c_{1}, \cdots, c_n)$ for $\chi_{\gamma}(x)=x^n+c_1x^{n-1}+\cdots + c_{n-1}x+c_n$ with $c_i\in F$.
The morphism $\rho_n$ is then representable as a morphism of schemes over $F$. 
Let $\M_{n,F}^{\ast}$ be the smooth locus of $\rho_n$. 
 It is non-empty and a nonsingular variety since the smooth locus of a morphism  is open (and non-empty for $\rho_n$). 

We define $G_{\gamma}$ to be $\rho_n^{-1}(\chi_{\gamma})$.
Since $\rho_n|_{\M_{n,F}^{\ast}}$ is smooth  and  $G_{\gamma}$ is a subvariety of  $\M_{n,F}^{\ast}$, 
$G_{\gamma}$ is also smooth over $F$ since smoothness is stable under base change.

\begin{definition}{\cite[Definition 2.1]{CKL}}\label{diff1}
We will define a differential  $\omega_{\chi{\gamma}}^{\mathrm{ld}}$ on $G_{\gamma}$
associated to $\omega_{\M_{n, \mathfrak{o}}}$ and $\omega_{\mathbb{A}^n_{\mathfrak{o}}}$.
Smoothness of the morphism $\rho_n : \M_{n,F}^{\ast} \rightarrow \mathbb{A}_F^n$ induces the following short exact sequence of locally free sheaves on $\M_{n,F}^{\ast}$ (cf.  \cite[Proposition II.5]{BRL}):
\begin{equation*}\label{eqshort}
0\rightarrow \rho_n^{\ast}\Omega_{\mathbb{A}_F^n\slash F}\rightarrow \Omega_{\M_{n,F}^{\ast}\slash F} \rightarrow \Omega_{\M_{n,F}^{\ast}\slash \mathbb{A}_F^n} \rightarrow 0.
\end{equation*}
This gives rise to an isomorphism
\begin{equation*}\label{eqtop}
\rho_n^{\ast}\left(\bigwedge\limits^{\mathrm{top}}\Omega_{\mathbb{A}_F^n\slash F}\right)\otimes
\bigwedge\limits^{\mathrm{top}}\Omega_{\M_{n,F}^{\ast}\slash \mathbb{A}_F^n}\simeq 
\bigwedge\limits^{\mathrm{top}}\Omega_{\M_{n,F}^{\ast}\slash F}.
\end{equation*}

Let $\omega_{\chi_{\gamma}}\in \bigwedge\limits^{\mathrm{top}}\Omega_{\M_{n,F}^{\ast}\slash \mathbb{A}_F^n}(\M_{n,F}^{\ast})$ be such that 
$\rho_n^{\ast}\omega_{\mathbb{A}^n_{\mathfrak{o}}}\otimes \omega_{\chi_{\gamma}}=\omega_{\M_{n, \mathfrak{o}}}|_{\M_{n,F}^{\ast}}$ and denote by $\omega_{\chi_{\gamma}}^{\mathrm{ld}}$  the restriction of $\omega_{\chi_{\gamma}}$ to $G_{\gamma}$.
We sometimes write $\omega_{\chi_{\gamma}}^{\mathrm{ld}}=\omega_{\M_{n, \mathfrak{o}}}/\rho_n^{\ast}\omega_{\mathbb{A}^n_{\mathfrak{o}}}$.
\end{definition}

As in \cite[Definition 2.2]{CKL}, we define the orbital integral as follows:

\begin{definition}\label{def:orbitalint}
    The  orbital integral for $\gamma$ and for an element $\mathbbm{1}_{D_{(k_1,\cdots,k_n)}}$ in the  spherical Hecke algebra, denoted by $\mathcal{SO}_{\gamma}(\mathbbm{1}_{D_{(k_1,\cdots,k_n)}})$, is defined to be
    \[
    \mathcal{SO}_{\gamma}(\mathbbm{1}_{D_{(k_1,\cdots,k_n)}})=\int_{G_\gamma(F)\cap D_{(k_1,\cdots,k_n)}}|\omega_{\chi_\gamma^{\mathrm{ld}}}|.
    \]
\end{definition}
Here `$\mathcal{S}$' stands for `$stable$', which comes from stable orbital integrals. 

\subsubsection{Quotient measure}\label{classical_measure}
On the other hand, the quotient measure, which will be defined below, is usually taken to define the  orbital integral in the literature. 
We will explain it closely following \cite[Section 2.2]{CKL}, which is  based on  \cite{Yun13}. 

Let
\[
\left\{
\begin{array}{l}
\textit{$\mathrm{T}_{\gamma}$ be the centralizer of  $\gamma$ via the conjugation in $\mathrm{GL}_{n,F}$ which is a  maximal torus};\\
\textit{$\mathrm{T}_c$ be the maximal compact subgroup of $\mathrm{T}_{\gamma}(F)$};\\
\textit{$dt$ be the Haar measure on $\mathrm{T}_{\gamma}(F)$ such that $\mathrm{vol}(dt, \mathrm{T}_c)=1$};\\
\textit{$dg$ be the Haar measure on $\mathrm{GL}_{n}(F)$ such that $\mathrm{vol}(dg, \mathrm{GL}_n(\mathfrak{o}))=1$};\\
\textit{$d\mu=\frac{dg}{dt}$ be the quotient measure defined on $\mathrm{T}_{\gamma}(F)\backslash \mathrm{GL}_n(F)$.}
\end{array}\right.
\]
The    orbital integral of \cite{Yun13} uses the quotient measure $d\mu$. We denote it by $\mathcal{SO}_{\gamma, d\mu}(\mathbbm{1}_{D_{(k_1,\cdots,k_n)}})$ to emphasize the role of the Haar measure $d\mu$.
This is formulated as follows:
\[
\mathcal{SO}_{\gamma, d\mu}(\mathbbm{1}_{D_{(k_1,\cdots,k_n)}})=\int_{\mathrm{T}_{\gamma}(F)\backslash \mathrm{GL}_n(F)} \mathbbm{1}_{D_{(k_1,\cdots,k_n)}}(g^{-1}\gamma g) d\mu(g).
\]

\subsubsection{Explicit description of $\mathrm{T}_{\gamma}$}\label{sectorus}
This subsection is taken from \cite[Section 2.2.2]{CKL}. 
We will reproduce the relevant part following \cite[Sections 4.1 and 4.8]{Yun13}.
Let
\[
\left\{
\begin{array}{l}
\textit{$B(\gamma)$ be an index set in bijection with the irreducible factors $\chi_{\gamma, i}(x)$ of $\chi_{\gamma}(x)$};\\
\textit{$F_i$ be the finite field extension of $F$ obtained by adjoining a root  of $\chi_{\gamma, i}(x)$.}
\end{array}\right.
\]
Then the centralizer $\mathrm{T}_{\gamma}$ is identified with
\[
\mathrm{T}_{\gamma}\cong \prod_{i\in B(\gamma)} \mathrm{Res}_{F_i/F}\mathbb{G}_m.
\]

Let $\mathrm{L}=\prod_{i\in B(\gamma)}\mathrm{L}_i$ be a Levi subgroup of $\mathrm{GL}_{n, F}$ such that 
\[
\left\{
\begin{array}{l}
\textit{$\mathrm{Res}_{F_i/F}\mathbb{G}_m$ is a maximal torus of $\mathrm{L}_i$ 
 (so $\mathrm{L}_i\cong \mathrm{GL}_{n_i, F}$ with $n_i=[F_i:F]$)};\\
 \textit{$\gamma=(\gamma_i)\in \mathrm{L}(F)$ with $\gamma_i\in \mathrm{L}_i(F)$};\\
\textit{$\mathrm{Res}_{F_i/F}\mathbb{G}_m$ is the centralizer of $\gamma_i$ in $\mathrm{L}_i$.}
\end{array}\right.
\]
Note that this $n_{i}=[F_{i}:F]$ is different to $n_{i}$ defined in \cite{Yun13}.
 Here, $Lie(\mathrm{G})$ is the Lie algebra of $\mathrm{G}$ for any algebraic group $\mathrm{G}$.
Due to the last condition in the above description of $\mathrm{L}$,  we can denote  $\mathrm{Res}_{F_i/F}\mathbb{G}_m$ by $\mathrm{T}_{\gamma_i}$.
Then $\mathrm{T}_{\gamma_i}$ has a natural integral structure given by $\mathrm{Res}_{\mathfrak{o}_{F_i}/\mathfrak{o}_F}\mathbb{G}_m$, which we also denote by $\mathrm{T}_{\gamma_i}$.
Note that $\mathrm{Res}_{\mathfrak{o}_{F_i}/\mathfrak{o}_F}\mathbb{G}_m$ is a smooth group scheme over $\mfo_F$ 
and that $\mathrm{Res}_{\mathfrak{o}_F}^{\mathfrak{o}_{F_i}}\mathbb{G}_m(\mfo_F)=\mfo_{F_i}^{\times}$ is the maximal compact subgroup of $\mathrm{T}_{\gamma_i}(F)=F_i^{\times}$.
Therefore, 
\[
\mathrm{T}_c = \prod_{i\in B(\gamma)} \mfo_{F_i}^{\times}.
\]

\subsubsection{Comparison of two normalizations}
 The difference between two measures $\omega_{\chi_{\gamma}}^{\mathrm{ld}}$ and $d\mu$ is described in  \cite[Proposition 3.29]{FLN} and is stated in \cite[Proposition 2.4]{CKL} for our normalization.

\begin{proposition}\cite[Proposition 3.29]{FLN} or \cite[Proposition 2.4]{CKL}\label{proptrans}
Suppose that $char(F) = 0$ or $char(F)>n$.
The difference between two  orbital integrals
$\mathcal{SO}_{\gamma}$ and $\mathcal{SO}_{\gamma, d\mu}$
 is described by the equation:
\[
\mathcal{SO}_{\gamma}(\mathbbm{1}_{D_{(k_1,\cdots,k_n)}})=
|\Delta_{\gamma}|^{1/2}\cdot
\frac{\#\mathrm{GL}_n(\kappa)q^{-\mathrm{dim}\mathrm{GL_n}}}
{\#\mathrm{T}_{\gamma}(\kappa)q^{-\mathrm{dim}\mathrm{T}_{\gamma}}}
\cdot \left(\prod_{i\in B(\gamma)}|\Delta_{F_i/F}|^{-1/2}\right)
\cdot \mathcal{SO}_{\gamma, d\mu}(\mathbbm{1}_{D_{(k_1,\cdots,k_n)}}).
\]
Here, $\Delta_{F_i/F}$ is the discriminant of the field extension $F_i/F$.
\end{proposition}

In the case that $\chi_{\gamma}(x)$ is irreducible, let $F_{\chi_{\gamma}}=F[x]/(\chi_{\gamma}(x))$ and let $\Delta_{F_{\chi_{\gamma}}/F}$ be the discriminant of the field extension $F_{\chi_{\gamma}}/F$.

\begin{proposition}{\cite[Proposition 2.5]{CKL}}\label{propserre}
If $\chi_{\gamma}(x)$ is irreducible, then 
$|\Delta_{\gamma}|$ and $\Delta_{F_{\chi_{\gamma}}/F}$ are related by the following equation:
\[
|\Delta_{F_{\chi_{\gamma}}/F}|^{1/2}=q^{S(\gamma)}\cdot |\Delta_{\gamma}|^{1/2},
\]
 where $S(\gamma)$ is the Serre invariant, which is the relative $\mfo$-length $[\mathfrak{o}_{F_{\chi_{\gamma}}}:\mathfrak{o}[x]/(\chi_{\gamma}(x))]$ 
 (see \cite[Section 2.1]{Yun13}).
\end{proposition}


\subsection{Linear algebraic interpretation}
This section is taken from \cite[Section 3.2]{CKL}.
Let us introduce the following notations and related facts:
\begin{itemize}
\item Let $L$ and $M$ be free $\mfo$-modules of rank $n$ in $L\otimes_\mfo F$.

\item Define a functor $\cL(L,M)$ on the category of flat $\mfo$-algebras to the category of sets such that 
\[
\cL(L,M)(R)=\{f\mid\textit{$f:L\otimes R\rightarrow M\otimes R $ is $R$-linear and surjective}\}
\]
for a flat $\mfo$-algebra $R$.
The functor $\cL(L,M)$ is then represented by an open subscheme of $\mathrm{Hom}_{\mfo}(L,M)$, which is an affine space over $\mfo$ of dimension $n^2$, so as to be smooth over $\mfo$.
Here, `open subscheme' structure is due to surjectivity in $\cL(L,M)(R)$. 

\item 
We can assign the $\mfo$-scheme structures on $\mathrm{Aut}(L)(\mfo)$ and $\mathrm{End}(L)(\mfo)$ by defining the functors on the category of flat $\mfo$-algebras to the category of  sets as follows:
\[
\mathrm{Aut}(L)(R)=\mathrm{Aut}(L\otimes R) ~~~~ \textit{   and   }  ~~~~~\End(L)(R)=\End_{R}(L\otimes R)
\]
for an $\mfo$-algebra $R$.
We sometimes use $\mathrm{End}(L)$ to stand for the set $\mathrm{End}(L)(\mfo)$ and use $\mathrm{Aut}(L)$ to stand for the set $\mathrm{Aut}(L)(\mfo)$ if there is no confusion.

\item We express $\M_{n}=\End(L)$ and $\mathrm{GL}_n=\mathrm{Aut}(L)$ so that 
 $\M_n(F)=\End(L)(F)$ and $\mathrm{GL}_n(F)=\mathrm{Aut}(L)(F)$.
Then  we may and do regard $\cL(L,M)(\mfo)$  as  an open subset  of $\mathrm{GL}_n(F)$, which is open in  $\M_{n}(F)$.

\item Define the orbit of $\gamma$ inside $\cL(L,M)(\mfo)$ as follows;
\[O_{\gamma, \cL(L,M)}=G_\gamma(F)\cap \cL(L,M)(\mfo)=\{f\in \cL(L,M)(\mfo) \mid \vpi_n(f)=\vpi_n(\gamma)\}\] 
where the intersection is taken inside $\End(L)(F)$ and define the volume of $O_{\gamma,\cL(L,M)}$ as follows; 
\[
\mathcal{SO}_{\gamma,\cL(L,M)}=\int_{O_{\gamma,\cL(L,M)}}|\omega_{\chi_\gamma}^{\mathrm{ld}}|.
\]
\end{itemize}
\begin{definition}{\cite[Definition 3.5]{CKL}}\label{def_type}
    The type of $M$ is defined to 
    be $(k_1,\cdots,k_n)$ where $k_i\in \mathbb{Z}$ and $ k_i\leq k_j$ for $i\leq j$, such that $M\subset \pi^{k_1} L$ satisfying
    \[
    (\pi^{k_1} L)/M\cong \mfo/\pi^{k_1-k_1}\mfo\oplus\mfo/\pi^{k_2-k_1}\mfo\oplus \cdots \oplus \mfo/\pi^{k_n-k_1}\mfo.
    \]
\end{definition}
For example, if $L$ is spanned by $(e_1, e_2)$ and if $M$ is spanned by $(1/\pi \cdot e_1, \pi \cdot  e_2)$, then the type of $M$ is $(-1, 1)$.

\begin{remark}
 The type defined in  \cite[Definition 3.5]{CKL} is slightly different from our definition.
 In loc. cit., $k_i$ is always positive, whereas we allow $k_i$ to be either zero or negative.
 
 For example, if $L$ is spanned by $(e_1, e_2)$ and if $M$ is spanned by $(e_1, \pi \cdot  e_2)$, then the type of $M$ is $(0, 1)$.
But in \cite[Definition 3.5]{CKL}, the type of $M$ was $(1)$, by ignoring $0$'s.
\end{remark}

\begin{lemma}{\cite[Lemma 3.6]{CKL}}\label{lemma36}
If the type of $M$ is the same as that of $M'$, then $\mathcal{SO}_{\gamma, \cL(L,M)}=\mathcal{SO}_{\gamma, \cL(L,M')}$.
\end{lemma}
\begin{proof}
When $k_1\geq 0$, it is proved in \cite[Lemma 3.6]{CKL}.
All arguments of loc. cit. hold when $k_1 < 0$. 
\end{proof}

This lemma enables us to define the following two notations:
\[
\left\{
\begin{array}{l l}
\mathcal{SO}_{\gamma, (k_1, \cdots, k_{n})}:=\mathcal{SO}_{\gamma, \cL(L,M)}\textit{ where the type of $M$ is $(k_1, \cdots, k_{n})$};\\
c_{(k_1, \cdots, k_{n})}:=\#\{M\subset L\otimes F \mid \textit{the type of $M$ is $(k_1, \cdots, k_{n})$}\}.
\end{array} \right.
\]
\begin{proposition}\label{prop:hecke_type}
We have
\[
D_{(k_1,\cdots,k_n)}=\bigsqcup_{M:type(M)=(k_1, \cdots, k_n)}\cL(L,M)(\mfo) ~~~~ 
 \textit{   and   }  ~~~~~   \mathcal{SO}_{\gamma}(\mathbbm{1}_{(D_{(k_1,\cdots,k_n)})})=c_{(k_1,\cdots,k_n)}\cdot \mathcal{SO}_{\gamma,{(k_1,\cdots,k_n)}}
\]    
\end{proposition}
\begin{proof}
Recall that we identify  $\mathrm{GL}_n(F)=\mathrm{Aut}(L\otimes F)$.
For  $X\in \mathrm{Aut}(L\otimes F)$, let $X(L)$ be the image of $X$ so that $X(L)$ is a free $\mfo$-module of rank $n$ in $L\otimes_{\mfo} F$. 
Then     $X\in D_{(k_1,\cdots,k_n)}$ if and only if the type of $X(L)$ is $(k_1,\cdots,k_n)$.
This directly yields that 
\[
D_{(k_1,\cdots,k_n)}=\bigsqcup_{M:type(M)=(k_1, \cdots, k_n)}\cL(L,M)(\mfo).
\]
Then Lemma  \ref{lemma36} yields the desired second statement.
\end{proof}

For a further use, we state the relation between $c_{(k_1, \cdots, k_{n})}$ and $c_{(k_1-k_1, \cdots, k_{n}-k_1)}$.
\begin{proposition}
Two integers $c_{(k_1, \cdots, k_{n})}$ and $c_{(k_1-k_1, \cdots, k_{n}-k_1)}$ are equal. 
\end{proposition}\label{prop:counting}
\begin{proof}
It suffices to show that  the following map is bijective:
\[
\{M\subset L\otimes_\mfo F \mid type(M)=(k_1, \cdots, k_{n})\} 
\rightarrow \{N\subset L\otimes_\mfo F \mid type(N)=(k_1-k_1, \cdots, k_{n}-k_1)\},\ M\mapsto \pi^{-k_1}M.
\]
The inverse map $N\mapsto \pi^{k_1}N$ is well-defined and the composite of two maps are the identity. 
This completes the proof.    
\end{proof}

\subsection{Reduction on $\mathcal{SO}_{\gamma,\mathrm{Hom}(L,\pi^{k}L)}$}

\begin{proposition}\label{prop_red}{\cite[Proposition 4.11]{CKL}}
Suppose that $M \left(\subset \pi^{k_1}L\right)$ is of type $(k_1,\cdots,k_n)$.
Suppose that 
\[
\chi_{\gamma}(\pi^{k_1} x)/\pi^{nk_1}=x^n+c_{1}^{(k_1)}x^{n-1}+\cdots +c_{n-1}^{(k_1)}x+c_n^{(k_1)}\in \mfo[x], ~~~~  \textit{ where  } ~~~~~ c_i^{(k)}:=c_i/\pi^{ik_1}\in \mfo.
\]
Let 
$\gamma^{(k_1)}$ be an element of $\End(L)(\mfo)$ whose characteristic polynomial is 
$\chi_{\gamma^{(k_1)}}(x):=\chi_{\gamma}(\pi^{k_1} x)/\pi^{nk_1}$.
Then we have
\[
\mathcal{SO}_{\gamma, (k_1,\cdots,k_n)}=q^{-k_1\cdot \frac{n(n-1)}{2}}\mathcal{SO}_{\gamma^{(k_1)}, (k_1-k_1,\cdots,k_n-k_1)}.
\]
\end{proposition}
This is proved in \cite[Proposition 4.11]{CKL} when $k_1\geq 0$. But all arguments in loc. cit. hold whenever $\chi_{\gamma^{(k_1)}}(x)$ has coefficients in $\mfo$, especially when $k_1<0$.

\section{Explicit formula for orbital integrals  when $\gamma$ is elliptic}\label{sec4}
We  collect a few facts  about $\mathcal{SO}_{\gamma}(\mathbbm{1}_{\mathrm{M}_n}(\mfo))$, taken  from \cite{CKL}.

\begin{proposition}\cite[Proposition 3.7]{CKL}\label{prop_stratification}
We have
    \[
    O_{\gamma,\mathrm{M}_n(\mfo)}=\bigsqcup_{\substack{(k_1, \cdots, k_{n}),  \\ \sum k_i=\mathrm{ord}(c_n),\ k_1\geq 0 }}
\Bigg(\bigsqcup_{\substack{M:type(M)= \\ (k_1, \cdots, k_{n})}} O_{\gamma, \cL(L,M)}\Bigg)\textit{ and so }
\mathcal{SO}_{\gamma}(\mathbbm{1}_{\mathrm{M}_n(\mfo)})= \sum\limits_{\substack{(k_1, \cdots, k_n),  \\ \sum k_i=\mathrm{ord}(c_n),\ k_1 \geq 0 }}
\mathcal{SO}_{\gamma}(\mathbbm{1}_{D_{(k_1,\cdots,k_n)}}),
\]
where $O_{\gamma,\mathrm{M}_n(\mfo)}=G_\gamma(F)\cap \mathrm{M}_n(\mfo)$.
\end{proposition}

\begin{lemma}\cite[Lemma 3.4]{CKL}\label{constantlemma}
For a constant matrix $c\in \mathrm{M}_n(\mfo)$,
we have $\mathcal{SO}_{\gamma}(\mathbbm{1}_{\mathrm{M}_n(\mfo)})=\mathcal{SO}_{\gamma+c}(\mathbbm{1}_{\mathrm{M}_n(\mfo)})$ along the bijection
$O_{\gamma, \mathrm{M}_n(\mfo)}\rightarrow O_{\gamma+c, \mathrm{M}_n(\mfo)}, ~~~~~~~~  X\mapsto X+c$.
\end{lemma}

From now on, we suppose that our regular and semisimple element $\gamma$ in $\mathrm{GL}_{n}(F)\cap \mathrm{M}_n(\mfo)$ is elliptic, equivalently $\chi_{\gamma}(x)\in \mfo[x]$ is irreducible. 

\begin{definition}\label{def:chigamma}
    For $a\in \mfo$, we define
    \begin{enumerate}
        \item 
 $\chi_{\gamma,a}(x):=\chi_\gamma(x+a)$ so that 
$\chi_{\gamma,a}(x)$ is the characterisitic polynomial of $\gamma-a\cdot I_n$;
\item  $d_a:=\ord (\chi_{\gamma,a}(0))$.
    \end{enumerate}
    \end{definition}
Recall that $F_{\chi_{\gamma}}=F[x]/(\chi_{\gamma}(x))$ and that $S(\gamma)=[\mathfrak{o}_{F_{\chi_{\gamma}}}:\mathfrak{o}[x]/(\chi_{\gamma}(x))]$ (cf. Proposition \ref{propserre}).

\begin{lemma}\label{prop_specialconst}
Let $n$ be a prime integer.
\begin{enumerate}    
    \item\cite[Proposition 5.1]{CKL} There exists $a\in \mfo$ such that 
\[
\left\{
\begin{array}{l l}
    d_{a}\equiv 0 \textit{ modulo }n   &\textit{ if $F_{\chi_{\gamma}}$ is unramified over $F$};\\
    d_{a}\not\equiv 0 \textit{ modulo }n  &\textit{ if $F_{\chi_{\gamma}}$ is (totally) ramified over $F$}
\end{array}
\right.
\]
and in the first case, the reduction of  $\frac{\chi_{\gamma, a }(\pi^{\frac{d_{a}}{n}} x)}{\pi^{d_a}}$ modulo $\pi$ is irreducible over $\kappa$.

\item\cite[Remark 6.4]{CKL} For  an element $a\in \mfo$ chosen as in the above, we have
the following relation between $d_a$ and $S(\gamma)$:
\[
S(\gamma)=\left\{
\begin{array}{l l}  
    \frac{d_a(n-1)}{2} &\textit{ if $F_{\chi_{\gamma}}$ is unramified over $F$};\\
    \frac{(d_a-1)(n-1)}{2} &\textit{ if $F_{\chi_{\gamma}}$ is ramified over $F$}.
\end{array}
\right.
\]
Thus $d_a$ is independent of the choice of $a$ in the above (1).
\end{enumerate}
\end{lemma}

\begin{remark}\label{remark:comp}
 The proof of \cite[Proposition 5.1]{CKL} yields that $d_a\geq \mathrm{ord}(\det \gamma)$.
 This will be used in Theorem \ref{mainthmgl3}.(3).
\end{remark}

\begin{proposition}\cite[Corollary 4.9]{CKL}\label{prop_k_ntype}
For $k_n>0$ such that $k_n=\mathrm{ord}(c_n)$, we have
\[\mathcal{SO}_{\gamma}(\mathbbm{1}_{D_{(0,\cdots,0,k_n)}})=\frac{\#\mathrm{GL}_{n}(\kappa)}{(q-1)\cdot q^{n^{2}-1}}.\]
\end{proposition}
Note that this is a reformulation of \cite[Corollary 4.9]{CKL} in the context of  Proposition \ref{prop:hecke_type}.

\subsection{Orbital integrals on the  spherical  Hecke algebra of $\mathrm{GL}_2$}\label{subsec4.1}
In this subsection, we will state our main result for $\mathrm{GL}_2$ in Theorem \ref{mainthm48} in terms of the geometric measure (cf. Section \ref{subsection:geometric}) and in Theorem \ref{result_2_dmu} in terms of the quotient measure (cf. Section \ref{classical_measure}).
Throughout this subsection, we denote $\mathrm{ord}(c_2)$ by $d$.

\begin{proposition}\cite[Theorem 5.4]{CKL}\label{thm_for_M_2}
    For an elliptic regular semisimple element $\gamma \in \mathrm{GL}_2(F)\cap \mathrm{M}_2(\mfo)$, the orbital integral for $\gamma$ and for $\mathbbm{1}_{\mathrm{M}_2(\mfo)}$ is as follows:
    \[
    \mathcal{SO}_{\gamma}(\mathbbm{1}_{\mathrm{M}_2(\mfo)})=\left\{\begin{array}{l l}
    \frac{q +1}{q}-\frac{2}{q^{S(\gamma)+1}} & \textit{if $F_{\chi_{\gamma}}/F$ is unramified};\\
 \frac{q+1}{q}-\frac{q+1}{q^{S(\gamma)+2}} & \textit{if $F_{\chi_{\gamma}}/F$ is ramified}.
\end{array}\right.
    \]
\end{proposition}

\begin{theorem}\label{mainthm48}
For an elliptic regular semisimple element $\gamma\in\mathrm{GL}_2(F)\cap \mathrm{M}_{2}(\mfo)$, the orbital integral for $\gamma$ and for $\mathbbm{1}_{D_{(k_1,k_2)}}$ is as follows:
    \begin{enumerate}
        \item In the case that $k_1+k_2\neq d$, we have
        \[
        \mathcal{SO}_{\gamma}(\mathbbm{1}_{D_{(k_1,k_2)}})=0.
        \]

        \item In the case that $k_1+k_2=d$ and $k_1 < \frac{d}{2}$ (so that 
        $(k_1,k_2)=(k_1,d-k_1)$), we have
        \[
        \mathcal{SO}_{\gamma}(\mathbbm{1}_{D_{(k_1,k_2)}})=
            q^{-k_1}\cdot \frac{q^2-1}{q^2}.
        \]
        \item In the case that $(k_1,k_2)=(\frac{d}{2},\frac{d}{2})$, we have
        \[
        \mathcal{SO}_{\gamma}(\mathbbm{1}_{D_{(\frac{d}{2},\frac{d}{2})}})=
        \left\{
        \begin{array}{l l}
            \frac{q+1}{q^{\frac{d}{2}+1}}-\frac{2}{q^{S(\gamma)+1}}&\textit{if $F_{\chi_\gamma}/F$ is unramified};\\
           \frac{q+1}{q^{\frac{d}{2}+1}}-\frac{q+1}{q^{S(\gamma)+2}}&\textit{if $F_{\chi_\gamma}/F$ is ramified}.
        \end{array}
        \right.
        \]
    \end{enumerate}
\end{theorem}
\begin{proof}
        Since $\mathrm{ord}(\det \gamma)=d$, $\mathcal{SO}_{\gamma}(\mathbbm{1}_{D_{(k_1,k_2)}})=0$ unless $k_1+k_2= d$.
        Suppose that $k_1+k_2=d$.
        By Propositions \ref{prop:hecke_type}-\ref{prop_red}, we have
        \[
        \mathcal{SO}_{\gamma}(\mathbbm{1}_{D_{(k_1,d-k_1)}})=q^{-k_1}\cdot \mathcal{SO}_{\gamma^{(k_1)}}(\mathbbm{1}_{D_{(0,d-2k_1)}}).
        \]
\begin{itemize}
\item         In the case \textit{(2)}, since $d-2k_1>0$, Proposition \ref{prop_k_ntype} yields that    \[
        \mathcal{SO}_{\gamma^{(k_1)}}(\mathbbm{1}_{D_{(0,d-2k_1)}})=\frac{q^2-1}{q^2}.\]

    \item         In the case \textit{(3)}, we have $d-2k_1=0$ so that   the determinant of $\gamma^{(k_1)}$ is a unit in $\mfo$. Thus we have
    \[
        \mathcal{SO}_{\gamma^{(k_1)}}(\mathbbm{1}_{D_{(0,0)}})
        =\mathcal{SO}_{\gamma^{(k_1)}}(\mathbbm{1}_{\mathrm{M}_2}(\mfo)).
        \]
        Proposition \ref{thm_for_M_2} then yields  that
        \[
        \mathcal{SO}_{\gamma}(\mathbbm{1}_{D_{(\frac{d}{2},\frac{d}{2})}})=q^{-\frac{d}{2}}\mathcal{SO}_{\gamma^{(\frac{d}{2})}}(\mathbbm{1}_{D_{(0,0)}})
        =\left\{\begin{array}{l l}
           q^{-\frac{d}{2}}(\frac{q+1}{q}-\frac{2}{q^{S(\gamma^{(\frac{d}{2})})+1}})  &\textit{if $F_{\chi_\gamma}/F$ is unramified};  \\
            q^{-\frac{d}{2}}(\frac{q+1}{q}-\frac{q+1}{q^{S(\gamma^{(\frac{d}{2})})+2}}) & \textit{if $F_{\chi_\gamma}/F$ is ramified}.
        \end{array}
        \right.
        \]
On the other hand,  $S(\gamma^{(d/2)})=S(\gamma) - \frac{d}{2}$ by Proposition \ref{propserre}.
Here we use the fact that $F_{\chi_{\gamma}}=F_{\chi_{\gamma^{(d/2)}}}$.
This completes the proof. 
\end{itemize}
\end{proof}

By the comparison between the geometric measure and the quotient measure in  Proposition \ref{proptrans}, we can rewrite the above formula in terms of the quotient measure.
\begin{theorem}\label{result_2_dmu}
Suppose that $char(F) = 0$ or $char(F)>2$.
For an elliptic regular semisimple element $\gamma\in\mathrm{GL}_2(F)\cap\mathrm{M}_n(\mfo)$, the orbital integral for $\gamma$ and for $\mathbbm{1}_{D_{(k_1,k_2)}}$ with respect to the quotient measure $d\mu$, defined in Section \ref{classical_measure}, is as follows:
   \begin{enumerate}
        \item In the case that $k_1+k_2\neq d$, we have
        \[
        \mathcal{SO}_{\gamma,d\mu}(\mathbbm{1}_{D_{(k_1,k_2)}})=0.
        \]

        \item In the case that $k_1+k_2=d$ and $k_1 < \frac{d}{2}$ (so that 
        $(k_1,k_2)=(k_1,d-k_1)$), we have
        \[
        \mathcal{SO}_{\gamma,d\mu}(\mathbbm{1}_{D_{(k_1,k_2)}})=
            \left\{
        \begin{array}{l l}
            (q+1)q^{S(\gamma)-k_1-1} & \textit{if $F_{\chi_\gamma/F}$ is unramified};\\
            q^{S(\gamma)-k_1} & \textit{if $F_{\chi_\gamma/F}$ is ramified}.
        \end{array}
        \right.
        \]
        \item In the case that $(k_1,k_2)=(\frac{d}{2},\frac{d}{2})$, we have
        \[
        \mathcal{SO}_{\gamma,d\mu}(\mathbbm{1}_{D_{(k_1,k_2)}})=
        \left\{
        \begin{array}{l l}
            \frac{q+1}{q-1}(q^{S(\gamma)-\frac{d}{2}}-1)+1&\textit{if $F_{\chi_\gamma}/F$ is unramified};\\
           \frac{1}{q-1}(q\cdot q^{S(\gamma)-\frac{d}{2}}-1)&\textit{if $F_{\chi_\gamma}/F$ is ramified}.
        \end{array}
        \right.
        \]
    \end{enumerate}
\end{theorem}
\begin{proof}
    By Proposition \ref{proptrans} together with  Proposition \ref{propserre}, we have 
    \[
    \mathcal{SO}_{\gamma,d\mu}(\mathbbm{1}_{D_{(k_1,k_2)}})=q^{S(\gamma)}\cdot \frac{\#\mathrm{T}_\gamma(\kappa)\cdot q^{4}}{\#\mathrm{GL}_2(\kappa)\cdot q^2}\cdot \mathcal{SO}_{\gamma}(\mathbbm{1}_{D_{(k_1,k_2)}}).
    \]
    By the description of $T_{\gamma}$ in Section \ref{sectorus}, we have
    \[
\#\mathrm{T}_\gamma(\kappa)=
        \left\{
        \begin{array}{l l}
 q^2-1 &\textit{if $F_{\chi_\gamma}/F$ is unramified};\\
      q^2-q &\textit{if $F_{\chi_\gamma}/F$ is ramified}.
        \end{array}
        \right.
        \]
                Combination of these two completes the proof. 
\end{proof}

\begin{remark}
\cite[Section 5.9]{Kot05} describes a formula when $\gamma$ is an elliptic regular semisimple element in $\mathrm{GL}_2(\mfo)$ so that $d=0$ and $k_2=-k_1$.
In \cite[Example 4.1]{Gor22}, Gordon explains the normalization used in Kottwitz's work:  Kottwitz uses the quotient measure $\frac{dg}{dt'}$ where $dg$ is the Haar measure on $\mathrm{GL}_2(F)$ which is the same as that defined in Section \ref{classical_measure} and where $dt'$ is the Haar measure on $\mathrm{T}_\gamma(F)$ normalized as follows:
\[
\mathrm{vol}(dt',\mathrm{T}_c)=\left\{
\begin{array}{l l}
   1  & \textit{if $F_{\chi_\gamma}/F$ is unramified}; \\
\frac{1}{2} & \textit{if $F_{\chi_\gamma}/F$ is ramified}
\end{array}\right.
~~~ \textit{ so that }  ~~~~~~   dt'=\left\{
\begin{array}{l l}
   dt  & \textit{if $F_{\chi_\gamma}/F$ is unramified}; \\
\frac{1}{2} dt & \textit{if $F_{\chi_\gamma}/F$ is ramified}.
\end{array}\right.
\]
Here $dt$ is the Haar measure on $\mathrm{T}_\gamma(F)$ which is defined in Section \ref{classical_measure}.
By accounting the difference between $dt'$ and $dt$, the special case of  Theorem \ref{result_2_dmu} when $d=0$ coincides with \cite[Section 5.9]{Kot05}.
Note that this comparison is based on  the fact that $S(\gamma)=d_{\gamma}$, where $d_{\gamma}$ is defined in  \cite[Section 5.9]{Kot05}.
\end{remark}


\subsection{Orbital integrals on the  spherical  Hecke algebra of $\mathrm{GL}_3$}\label{subsec4.2}
In this subsection, we will state our main result for $\mathrm{GL}_3$ in Theorem \ref{mainthmgl3} in terms of the geometric measure (cf. Section \ref{subsection:geometric}) and in Theorem \ref{mainthmgl3q} in terms of the quotient measure (cf. Section \ref{classical_measure}).
As in the previous subsection,  $\mathrm{ord}(c_3)$ is denoted  by $d$. 
Note that if $d>0$, then the reduction of $\chi_{\gamma}(x)$ modulo $\pi$ is $x^3$ since $\gamma$ is elliptic, equivalently $\chi_{\gamma}(x)$ is irreducible. 

\begin{proposition}\cite[Theorem 6.1]{CKL}\label{result_3}
Suppose that $d>0$.
For an elliptic regular semisimple element $\gamma\in\mathrm{GL}_3(F)\cap \mathrm{M}_{3}(\mfo)$, 
the formula for $\mathcal{SO}_{\gamma}(\mathbbm{1}_{D_{(0,k_2, k_3)}})$ with $k_2+k_3=d$ is given as follows:
\[
\mathcal{SO}_{\gamma}(\mathbbm{1}_{D_{(0,k_2, k_3)}})=\left\{
  \begin{array}{l l}
(q^{3}-1)(q^{2}-1)q^{k_{3}}\cdot \frac{(k_2+1)q-(k_2-1)}{q^{6+d}}   & \textit{if $0<k_2<\lceil\frac{d}{3}\rceil$ and $k_2< k_3$};\\
(q^{3}-1)(q^{2}-1)q^{k_{3}}\cdot \frac{(k_3-k_2+1)q-(k_3-k_2-1)}{q^{6+d}}   & \textit{if $\lfloor \frac{d}{3}\rfloor<k_2<k_3$};\\
\frac{(q^{3}-1)(q^{2}-1)}{q^{5+d}}\cdot q^{k_{2}}   & \textit{if $k_2=k_3$};\\
(q^{3}-1)(q^{2}-1)q^{\frac{2d}{3}}\cdot \frac{(\frac{d}{3}+1)q-(\frac{d}{3}-1)}{q^{6+d}} & \textit{if $(k_2, k_3)=(\frac{d}{3}, \frac{2d}{3})$},
\end{array} \right.
\]
where in the fourth case we suppose that the reduction of $\chi_\gamma(\pi^{\frac{d}{3}}x)/\pi^d$ modulo $\pi$ is irreducible over $\kappa$.
\end{proposition}

\begin{proposition}\cite[Theorem 6.5]{CKL}\label{thm_for_M_3}
    For an elliptic regular semisimple element $\gamma \in\mathrm{GL}_3(F)\cap\mathrm{M}_3(\mfo)$, the orbital integral for $\gamma$ and for $\mathbbm{1}_{\mathrm{M}_3(\mfo)}$ is as follows:
    \[
    \mathcal{SO}_{\gamma}(\mathbbm{1}_{\mathrm{M}_3(\mfo)})=\left\{\begin{array}{l l}
\frac{(q+1)(q^{2}+q+1)}{q^{3}}-\frac{3(q^{2}+q+1)}{q^{d'+3}}+\frac{3}{q^{3d'+3}}&\textit{if }F_{\chi_{\gamma}}/F\textit{ is unramified so that }S(\gamma)=3d';\\
\frac{(q+1)(q^{2}+q+1)}{q^{3}}-\frac{(2q+1)(q^{2}+q+1)}{q^{d'+4}}+\frac{q^{2}+q+1}{q^{3d'+5}} & \textit{if }F_{\chi_{\gamma}}/F\textit{ is ramified and }S(\gamma)=3d';\\
\frac{(q+1)(q^{2}+q+1)}{q^{3}}-\frac{(q+2)(q^{2}+q+1)}{q^{d'+4}}+\frac{q^{2}+q+1}{q^{3d'+6}} & \textit{if }F_{\chi_{\gamma}}/F\textit{ is ramified and }S(\gamma)=3d'+1.\\
\end{array}\right.
    \]
\end{proposition}
\begin{theorem}\label{mainthmgl3}
    For an elliptic regular semisimple element $\gamma \in \mathrm{GL}_3(F)\cap \mathrm{M}_{3}(\mfo)$, the orbital integral for $\gamma$ and for $\mathbbm{1}_{D_{(k_1,k_2,k_3)}}$ is as follows:
\begin{enumerate}
    \item In the case that $k_1+k_2+k_3\neq d$, 
$    \mathcal{SO}_{\gamma}(\mathbbm{1}_{D_{(k_1,k_2,k_3)}})=0$.
    
\item In the case that $k_1+k_2+k_3=d$ and $k_1<k_2<k_3$ with $k_2\neq \frac{d}{3}$, 
we have
    \[
    \mathcal{SO}_{\gamma}(\mathbbm{1}_{D_{(k_1,k_2,k_3)}})=q^{k_3-k_1-d-6}(q^3-1)(q^2-1)\cdot\left\{\begin{array}{l l}
((k_2-k_1+1)q-(k_2-k_1-1))&\textit{if }k_1< k_2<\frac{d}{3}<k_3;\\
((k_3-k_2+1)q-(k_3-k_2-1)) & \textit{if }k_1<\frac{d}{3}<k_2<k_3,
\end{array}\right.
    \]

\item In the case that $k_1+k_2+k_3=d$ and $k_1<\frac{d}{3}=k_2<k_3$, we have 
\[
    \mathcal{SO}_{\gamma}(\mathbbm{1}_{D_{(k_1,k_2,k_3)}})=
    \frac{(q^3-1)(q^2-1)}{q^{2k_1+d'+6}}\left(\left(\frac{d}{3}-k_1+1\right)q^{d'-\frac{d}{3}+1}-\left(\frac{d}{3}-k_1-1\right)q^{d'-\frac{d}{3}}+3q\frac{q^{d'-\frac{d}{3}}-1}{q-1}+\varepsilon(\gamma)\right),
    \]
    where $d'=\lfloor\frac{S(\gamma)}{3} \rfloor$ and
    $
    \varepsilon(\gamma)=\left\{\begin{array}{l l}
    0&\textit{if $F_{\chi_\gamma}/F$ is unramified and $S(\gamma)=3d'$};\\
    1& \textit{if $F_{\chi_\gamma}/F$ is ramified and $S(\gamma)=3d'$};\\
    2& \textit{if $F_{\chi_\gamma}/F$ is ramified and $S(\gamma)=3d'+1$}.
    \end{array}\right.
    $
\textit{ }

\item In the case that $k_1+k_2+k_3=d$ and exactly two of $k_i$'s are equal, 
we have 
   \[
    \mathcal{SO}_{\gamma}(\mathbbm{1}_{D_{(k_1,k_2,k_3)}})=\left\{\begin{array}{l l}
q^{k_2-k_1-d-5}(q^3-1)(q^2-1) & \textit{if }k_1< k_2=k_3;\\
q^{-3k_1-5}(q^3-1)(q^2-1) & \textit{if }k_1=k_2<k_3,
\end{array}\right.
    \]

    \item 
    In the case that $k_1=k_2=k_3=d/3$, 
 we have
    \[
    \mathcal{SO}_{\gamma}(\mathbbm{1}_{D_{(k_1,k_2,k_3)}})=
    \left\{\begin{array}{l l}
    \frac{(q+1)(q^{2}+q+1)}{q^{3+d}}-\frac{3(q^{2}+q+1)}{q^{d'+\frac{2d}{3}+3}}+\frac{3}{q^{3d'+3}} &\textit{if $F_{\chi_\gamma}/F$ is unramified and $S(\gamma)=3d'$};\\
    \frac{(q+1)(q^{2}+q+1)}{q^{3+d}}-\frac{(2q+1)(q^{2}+q+1)}{q^{d'+\frac{2d}{3}+4}}+\frac{q^{2}+q+1}{q^{3d'+5}} & \textit{if $F_{\chi_\gamma}/F$ is ramified and $S(\gamma)=3d'$};\\
    \frac{(q+1)(q^{2}+q+1)}{q^{3+d}}-\frac{(q+2)(q^{2}+q+1)}{q^{d'+\frac{2d}{3}+4}}+\frac{q^{2}+q+1}{q^{3d'+6}} & \textit{if $F_{\chi_\gamma}/F$ is ramified and $S(\gamma)=3d'+1$}.
    \end{array}\right.
    \]
\end{enumerate}
\end{theorem}
\begin{proof}
Before proving our formula, we remark about the structure of the proof. 
(1) is obvious and (2), (4), (5) are simple applications of Propositions \ref{prop_k_ntype}, \ref{result_3}, and \ref{thm_for_M_3}, combined with Propositions \ref{prop:hecke_type}- \ref{prop_red}.
(3) requires an extra idea.
In the following, we will prove our formulas case by case with details.    
    
    \begin{enumerate}
  \item  Since $\mathrm{ord}(\det \gamma)=d$, $\mathcal{SO}_{\gamma}(\mathbbm{1}_{D_{(k_1,k_2,k_3)}})=0$ unless $k_1+k_2+k_3=d$.

From now on in the proof, we suppose that $k_1+k_2+k_3=d$.
 By Propositions \ref{prop:hecke_type}- \ref{prop_red}, we have
\begin{equation}\label{equation:gl3}
        \mathcal{SO}_{\gamma}(\mathbbm{1}_{D_{(k_1,k_2,k_3)}})=q^{-3k_1}\cdot \mathcal{SO}_{\gamma^{(k_1)}}(\mathbbm{1}_{D_{(0,k_2-k_1,k_3-k_1)}}).
    \end{equation}


\item Suppose that  that $k_1+k_2+k_3=d$ and that $k_1<k_2<k_3$ with $k_2\neq \frac{d}{3}$. 
By Equation (\ref{equation:gl3}), we apply the first of Proposition \ref{result_3} to $\mathcal{SO}_{\gamma^{(k_1)}}(\mathbbm{1}_{D_{(0,k_2-k_1,k_3-k_1)}})$ when $k_2-k_1<\frac{d}{3}-k_1<k_3-k_1$ and the second of Proposition \ref{result_3} when $\frac{d}{3}-k_1<k_2-k_1<k_3-k_1$. 
Then we obtain the desired result.
\\

\item Suppose that $k_1+k_2+k_3=d$ and that $k_1<k_2=\frac{d}{3}<k_3$. 
By Equation (\ref{equation:gl3}), we will work with 
$\mathcal{SO}_{\gamma^{(k_1)}}(\mathbbm{1}_{D_{(0,\frac{d}{3}-k_1,k_3-k_1)}})$.
Note that  $\mathrm{ord}(\det \gamma^{(k_1)})=d-3k_1$ and thus  $k_3-k_1=2(\frac{d}{3}-k_1)$.

To ease the notation let  $d^{\dag}:=d-3k_1>0$. We then have $\frac{d}{3}-k_1=\frac{d^{\dag}}{3}$ and $k_3-k_1=\frac{2d^{\dag}}{3}$.
Then we need to analyze $\mathcal{SO}_{\gamma^{(k_1)}}(\mathbbm{1}_{D_{(0,\frac{d^{\dag}}{3}, \frac{2d^{\dag}}{3})}})$.
We emphasize that it is not possible to use the fourth equation in Proposition \ref{result_3} due to the absence of  an extra condition required. 

We prove our formula step by step in the following.

\begin{enumerate}
    \item 
Choose $a\in \mfo$ and the integer $d^\dag_{a}$ associated to $\gamma^{(k_1)}$ in Lemma  \ref{prop_specialconst}.(1) so that 
 the reduction of $\chi_{\gamma^{(k_1)},a}(x):=\chi_{\gamma^{(k_1)}}(x+a)$ modulo $\pi$ is either irreducible or $x^3$, and so that $d^\dag_{a}=\mathrm{ord}\left(\chi_{\gamma^{(k_1)},a}(0)\right)$. Since the reduction of $\chi_{\gamma^{(k_1)}}(x)$ modulo $\pi$ is $x^3$ (because $d^{\dag}>0$), the reduction of $a$ modulo $\pi$ is $0$, equivalently $a\in \left(\pi\right)$.
Recall that $\chi_{\gamma^{(k_1)},a}(x)$ is the characteristic polynomial of $\gamma^{(k_1)}-a\cdot I_3$ in Definition \ref{def:chigamma}.
\item We claim that
    \begin{equation}\label{eq:special_type1}
    \sum_{j=1}^{\lfloor d^{\dag}/2 \rfloor}\mathcal{SO}_{\gamma^{(k_1)}}(\mathbbm{1}_{D_{(0,j,d^{\dag}-j)}})=
    \sum_{i=1}^{\lfloor d^\dag_{a}/2 \rfloor}\mathcal{SO}_{\gamma^{(k_1)}-a\cdot I_3}(\mathbbm{1}_{D_{(0,i,d^\dag_{a} -i)}}).
    \end{equation}

    Recall from  Lemma \ref{constantlemma} that there exists a bijection 
    $$\iota_a:
    O_{\gamma^{(k_1)},\mathrm{M}_3(\mfo)}
    \longrightarrow O_{\gamma^{(k_1)}-a\cdot I_3,\mathrm{M}_3(\mfo)}, ~~~~ \ X\mapsto X-a\cdot I_3$$ which yields the identity  
     $\mathcal{SO}_{\gamma^{(k_1)}}(\mathbbm{1}_{\mathrm{M}_3(\mfo)})=\mathcal{SO}_{\gamma^{(k_1)}-a\cdot I_3}(\mathbbm{1}_{\mathrm{M}_3(\mfo)})$. 
Then by Proposition \ref{prop:hecke_type}, it suffices to prove that the map $\iota_a$ yields a bijection: $$\iota_a: \bigsqcup\limits_{j=1}^{\lfloor d^{\dag}/2 \rfloor}O_{\gamma^{(k_1)}, D_{(0,j,d^\dag-j)}} \longrightarrow \bigsqcup\limits_{i=1}^{\lfloor d^\dag_{a}/2 \rfloor}\mathcal{O}_{\gamma^{(k_1)}-a\cdot I_3,  D_{(0,i,d^\dag_{a} -i)}},$$  
where $O_{\gamma^{(k_1)}, D_{(0,j,d^\dag-j)}}=G_{\gamma^{(k_1)}}(F)\cap D_{(0,j,d^\dag-j)}$ and 
the same principle for $O_{\gamma^{(k_1)}-a\cdot I_3, D_{(0,i,d^\dag_a-i)}}$.
Choose $X\in \bigsqcup\limits_{j=1}^{\lfloor d^{\dag}/2 \rfloor}O_{\gamma^{(k_1)}, D_{(0,j,d^\dag-j)}}$. To show that $\iota_a(X)$ is contained in $\bigsqcup\limits_{i=1}^{\lfloor d^\dag_{a}/2 \rfloor}\mathcal{O}_{\gamma^{(k_1)}-a\cdot I_3,  D_{(0,i,d^\dag_{a} -i)}}$, it suffices to show that the rank of the reduction of  $X-a\cdot I_3$ modulo $\pi$ is $1$ since the complement of this set inside $O_{\gamma,\mathrm{M}_3(\mfo)}$ is characterized by the condition that the rank of the reduction modulo $\pi$ is either $0$ or $2$.
This follows from the above step (a), that is,  $a\in (\pi)$.
Therefore $\iota_a$ is injective.
The same principle holds for the inverse of $\iota_a$, which sends $Y$ to $Y+a\cdot I_n$. 
Thus $\iota_a$ is bijective.

\item We compute the right hand side of  Equation (\ref{eq:special_type1}) using Proposition \ref{result_3}. More precisely, we use
\[
\left\{\begin{array}{l l}
\textit{the first of Proposition \ref{result_3}} &\textit{if $1\leq i<\frac{d^\dag_a}{3}$};\\
\textit{the second of Proposition \ref{result_3}} &\textit{if $\frac{d^\dag_a}{3}<i\leq\lfloor \frac{d^\dag_a}{2} \rfloor$ and $i\neq \frac{d^\dag_a}{2}$};\\
\textit{the third of Proposition \ref{result_3}} &\textit{if $i=\frac{d^\dag_a}{2}$};\\
\textit{the fourth of Proposition \ref{result_3}} &\textit{if $i=\frac{d^\dag_a}{3}$}.
    \end{array}\right.
\]
Here when $i=\frac{d^\dag_a}{3}$ (which occurs exactly when $F_{\chi_{\gamma^{(k_1)}}}/F$ is unramified), the fourth of Proposition \ref{result_3} is applicable
since the characteristic polynomial of $\gamma^{(k_1)}-a\cdot I_3$, which is $\chi_{\gamma^{(k_1)},a}(x)$, satisfies the extra condition of loc. cit. by Lemma \ref{prop_specialconst}.(1). To sum up, we have the following:
    \begin{equation}\label{eq:specialtype_12}
    \sum_{i=1}^{\lfloor d^\dag_a/2 \rfloor}\mathcal{SO}_{\gamma^{(k_1)}-a\cdot I_3}(\mathbbm{1}_{D_{(0,i,d^\dag_a -i)}})=
    \frac{(q^3-1)(q^2-1)}{q^{6+d^\dag_a}}\left(2q^{d^\dag_a}+3q^{d^\dag_a-1}+\cdots+3q^{d^\dag_a-\lfloor \frac{d^\dag_a}{3} \rfloor+1}+\left(d^\dag_a -3\lfloor \frac{d^\dag_a}{3} \rfloor\right)q^{d^\dag_a-\lfloor \frac{d^\dag_a}{3} \rfloor}\right).
    \end{equation}

To simplify the above more, we let $\varepsilon(\gamma):=d^\dag_a-3\lfloor \frac{d^\dag_a}{3}\rfloor$.
Note that     Lemma \ref{prop_specialconst}.\textit{(2)} implies that  
    $S(\gamma^{(k_1)})=\left\{\begin{array}{l l}
   d^\dag_a &\textit{if $F_{\chi_{\gamma^{(k_1)}}}/F$ is unramified};\\
        d^\dag_a-1&\textit{if $F_{\chi_{\gamma^{(k_1)}}}/F$ is ramified} 
    \end{array}\right.$
    and that Proposition \ref{propserre} implies that $S(\gamma^{(k_1)})=S(\gamma)-3k_1$ since  $F_{\chi_{\gamma^{(k_1)}}}= F_{\chi_\gamma}$.
    Therefore, we have
    \[
    \varepsilon(\gamma)=\left\{\begin{array}{l l}
    0&\textit{if $F_{\chi_\gamma}/F$ is unramified and $S(\gamma)=3d'$};\\
    1& \textit{if $F_{\chi_\gamma}/F$ is ramified and $S(\gamma)=3d'$};\\
    2& \textit{if $F_{\chi_\gamma}/F$ is ramified and $S(\gamma)=3d'+1$}.
    \end{array}\right.
    \]
    
\item     We compute the left side of Equation (\ref{eq:special_type1}) except for $\mathcal{SO}_{\gamma^{(k_1)}}(\mathbbm{1}_{D_{(0,\frac{d^\dag}{3},\frac{2d^\dag}{3})}})$ 
 using Proposition \ref{result_3}  as follows:
    \begin{equation}\label{specialtype_23}
    \sum_{\substack{1\leq j\leq \lfloor d^{\dag}/2 \rfloor;\\ j\neq d^{\dag}/3}}\mathcal{SO}_{\gamma^{(k_1)}}(\mathbbm{1}_{D_{(0,j,d^{\dag}-j)}})
    =\frac{(q^3-1)(q^2-1)}{q^{6+d^{\dag}}}\left(2q^{d^{\dag}}+3q^{d^{\dag}-1}+\cdots+3q^{\frac{2d^{\dag}}{3}+1}-q^{\frac{2d^{\dag}}{3}}((\frac{d^{\dag}}{3}+1)q-(\frac{d^{\dag}}{3}-1))\right).
    \end{equation}

\item Finally, we obtain the formula for $\mathcal{SO}_{\gamma^{(k_1)}}(\mathbbm{1}_{D_{(0,\frac{d^\dag}{3},\frac{2d^\dag}{3})}})$ by subtracting Equation (\ref{specialtype_23}) from Equation (\ref{eq:specialtype_12}).
Here    the factor $3q\frac{q^{d'-\frac{d}{3}}-1}{q-1}$ comes from the difference between two geometric series appearing in Equations (\ref{eq:specialtype_12}) and (\ref{specialtype_23}).
We note that $d'\left(=\lfloor\frac{S(\gamma)}{3} \rfloor\right)\geq \frac{d}{3}$ since $d^\dag_a\geq d^\dag$ by Remark \ref{remark:comp}, where $d^\dag=d-3k_1$ and $d^\dag_a=3d'-3k_1+\varepsilon(\gamma)$ (cf. the paragraph following Equation (\ref{eq:specialtype_12})).
\\
\end{enumerate}

\item Suppose that $k_1+k_2+k_3=d$ and that exactly two of $k_i$'s are equal.
For the first case when $k_1< k_2=k_3$, 
by Equation (\ref{equation:gl3}), we apply the third of  Proposition \ref{result_3} to $\mathcal{SO}_{\gamma^{(k_1)}}(\mathbbm{1}_{D_{(0,k_2-k_1,k_3-k_1)}})$ when $0<k_2-k_1=k_3-k_1$.
Then we obtain the desired result.

For the second case when $k_1=k_2<k_3$, we use $\mathcal{SO}_{\gamma^{(k_1)}}(\mathbbm{1}_{D_{(0,0,k_3-k_1)}})=\frac{(q^3-1)(q^2-1)}{q^5}$ by Proposition \ref{prop_k_ntype}.
Then Equation (\ref{equation:gl3}) yields the desired formula.
\\

    \item Suppose that $k_1+k_2+k_3=d$ and that  $k_1=k_2=k_3=\frac{d}{3}$.  
     Equation (\ref{equation:gl3}) is written as $        \mathcal{SO}_{\gamma}(\mathbbm{1}_{D_{(\frac{d}{3},\frac{d}{3},\frac{d}{3})}})=q^{-d}\cdot \mathcal{SO}_{\gamma^{(\frac{d}{3})}}(\mathbbm{1}_{D_{(0,0,0)}})$.
     Since $\mathcal{SO}_{\gamma^{(\frac{d}{3})}}(\mathbbm{1}_{D_{(0,0,0)}})=\mathcal{SO}_{\gamma^{(\frac{d}{3})}}(\mathbbm{1}_{\mathrm{M}_3(\mfo)})$, 
Proposition \ref{thm_for_M_3} yields the desired result. 
Here we use   that   $S(\gamma^{(\frac{d}{3})})=S(\gamma)-d$ by Proposition \ref{propserre} since  $F_{\chi_{\gamma}}=F_{\chi_{\gamma^{(\frac{d}{3})}}}$.
    \end{enumerate}
\end{proof}

By the comparison between the geometric measure and the quotient measure in  Proposition \ref{proptrans}, we can rewrite the above formula in terms of the quotient measure.

\begin{theorem}\label{mainthmgl3q}
Suppose that $char(F) = 0$ or $char(F)>3$.
    For an elliptic regular semisimple element $\gamma \in \mathrm{GL}_3(\mfo)$, the orbital integral for $\gamma$ and for $\mathbbm{1}_{D_{(k_1,k_2,k_3)}}$ with respect to the measure $d\mu$, defined in Section \ref{classical_measure}, is as follows:
    \begin{enumerate}
        \item In the case that $k_1+k_2+k_3\neq d$, we have
    \[
    \mathcal{SO}_{\gamma,d\mu}(\mathbbm{1}_{D_{(k_1,k_2,k_3)}})=0.
    \]
    \item In the case that $k_1+k_2+k_3=d$
    and $k_1<k_2<k_3$ with $k_2\neq \frac{d}{3}$, we have the following formulas:
    \begin{itemize}
        \item If  $k_1< k_2<\frac{d}{3}<k_3$, then
       \[
    \mathcal{SO}_{\gamma,d\mu}(\mathbbm{1}_{D_{(k_1,k_2,k_3)}})=
    \left\{\begin{array}{l l}
     q^{S(\gamma)+k_3-k_1-d-3}(q^2+q+1)((k_2-k_1+1)q-(k_2-k_1-1))&\textit{if $F_{\chi_\gamma}/F$ is unramified};\\
     q^{S(\gamma)+k_3-k_1-d-1}((k_2-k_1+1)q-(k_2-k_1-1))& \textit{if $F_{\chi_\gamma}/F$ is ramified}.
    \end{array}\right.
    \] 
    \item
    If $k_1<\frac{d}{3}<k_2<k_3$, then
    \[
    \mathcal{SO}_{\gamma,d\mu}(\mathbbm{1}_{D_{(k_1,k_2,k_3)}})=
    \left\{\begin{array}{l l}
     q^{S(\gamma)+k_3-k_1-d-3}(q^2+q+1)((k_3-k_2+1)q-(k_3-k_2-1))&\textit{if $F_{\chi_\gamma}/F$ is unramified};\\
     q^{S(\gamma)+k_3-k_1-d-1}((k_3-k_2+1)q-(k_3-k_2-1))& \textit{if $F_{\chi_\gamma}/F$ is ramified}.
    \end{array}\right.
    \]
    \end{itemize}
    \item In the case that $k_1+k_2+k_3=d$ and $k_1<\frac{d}{3}=k_2<k_3$,
we have the following formulas:
\[
    \mathcal{SO}_{\gamma,d\mu}(\mathbbm{1}_{D_{(k_1,k_2,k_3)}})=
    \left\{\begin{array}{l l}
 q^{2d'-2k_1-3}(q^2+q+1)\cdot \left(\mathcal{F}(d,d',k_1)+\varepsilon(\gamma)\right) &\textit{if $F_{\chi_\gamma}/F$ is unramified};\\
     q^{S(\gamma)-2k_1-d'-1}\cdot \left(\mathcal{F}(d,d',k_1)+\varepsilon(\gamma)\right) & \textit{if $F_{\chi_\gamma}/F$ is ramified},
    \end{array}\right.
    \]
    where $\mathcal{F}(d,d',k_1)=\left(\frac{d}{3}-k_1+1\right)q^{d'-\frac{d}{3}+1}-\left(\frac{d}{3}-k_1-1\right)q^{d'-\frac{d}{3}}+3q\cdot\frac{q^{d'-\frac{d}{3}}-1}{q-1}$. 
  For the relation between $d'$ and $S(\gamma)$, and the notion of $\varepsilon(\gamma)$, see Theorem \ref{mainthmgl3}.(3).

    \item In the case that $k_1+k_2+k_3=d$ and exactly two of $k_i$'s are equal, we have the following formulas:
    \begin{itemize}
    \item If $k_1<k_2=k_3$, then
    \[
    \mathcal{SO}_{\gamma,d\mu}(\mathbbm{1}_{D_{(k_1,k_2,k_3)}})=
    \left\{\begin{array}{l l}
     q^{S(\gamma)+k_2-k_1-d-2}(q^2+q+1)&\textit{if $F_{\chi_\gamma}/F$ is unramified};\\
     q^{S(\gamma)+k_2-k_1-d}& \textit{if $F_{\chi_\gamma}/F$ is ramified}.
    \end{array}\right.
    \]
    \item
    If  $k_1=k_2<k_3$, then
    \[
    \mathcal{SO}_{\gamma,d\mu}(\mathbbm{1}_{D_{(k_1,k_2,k_3)}})=
    \left\{\begin{array}{l l}
     q^{S(\gamma)-3k_1-2}(q^2+q+1)&\textit{if $F_{\chi_\gamma}/F$ is unramified};\\
     q^{S(\gamma)-3k_1}& \textit{if $F_{\chi_\gamma}/F$ is ramified}.
    \end{array}\right.
    \]
    \end{itemize}

         \item In the case that $k_1=k_2=k_3=d/3$, we have 
        \[
    \mathcal{SO}_{\gamma}(\mathbbm{1}_{D_{(k_1,k_2,k_3)}})=
    \left\{\begin{array}{l l}
    \frac{q^{3d'-d}(q^{2}+q+1)}{(q-1)^2}-\frac{3q^{2d'-\frac{2d}{3}}(q^2+q+1)}{(q+1)(q-1)^2}+\frac{3}{(q+1)(q-1)^2} &\textit{if $F_{\chi_\gamma}/F$ is unramified and $S(\gamma)=3d'$};\\
    \frac{q^{3d'-d+2}}{(q-1)^2}-\frac{q^{2d'-\frac{2d}{3}+1}(2q+1)}{(q+1)(q-1)^2}+\frac{1}{(q+1)(q-1)^2} & \textit{if $F_{\chi_\gamma}/F$ is ramified and $S(\gamma)=3d'$};\\
    \frac{q^{3d'-d+3}}{(q-1)^2}-\frac{q^{2d'-\frac{2d}{3}+2}(q+2)}{(q+1)(q-1)^2}+\frac{1}{(q+1)(q-1)^2} & \textit{if $F_{\chi_\gamma}/F$ is ramified and $S(\gamma)=3d'+1$}.
    \end{array}\right.
    \]
\end{enumerate}
\end{theorem}

 \begin{proof}
    By Proposition \ref{proptrans} together with Proposition \ref{propserre}, we have 
    \[
    \mathcal{SO}_{\gamma,d\mu}(\mathbbm{1}_{D_{(k_1,k_2,k_3)}})=q^{S(\gamma)}\cdot \frac{\#\mathrm{T}_\gamma(\kappa)\cdot q^{9}}{\#\mathrm{GL}_3(\kappa)\cdot q^3}\cdot \mathcal{SO}_{\gamma}(\mathbbm{1}_{D_{(k_1,k_2,k_3)}}).
    \]
        By the description of $T_{\gamma}$ in Section \ref{sectorus}, we have
    \[
\#\mathrm{T}_\gamma(\kappa)=
        \left\{
        \begin{array}{l l}
 q^3-1 &\textit{if $F_{\chi_\gamma}/F$ is unramified};\\
      q^3-q^2 &\textit{if $F_{\chi_\gamma}/F$ is ramified}.
        \end{array}
        \right.
        \]
                Combination of these two completes the proof. 
\end{proof}

\bibliographystyle{alpha}
\bibliography{References}

\end{document}